\documentclass{amsart}%
\usepackage{amssymb}
\usepackage{amsaddr}
\usepackage{amsmath}
\usepackage{amsfonts}
\usepackage[style=numeric, maxbibnames = 99]{biblatex}
\usepackage{tikz-cd}
\usepackage[a4paper, margin=1.5in]{geometry}
\usepackage{graphicx}%
\providecommand{\U}[1]{\protect\rule{.1in}{.1in}}
\newtheorem{theorem}{Theorem}[section]
\theoremstyle{plain}

\newtheorem{corollary}[theorem]{Corollary}

\newtheorem{definition}[theorem]{Definition}
\newtheorem{example}[theorem]{Example}

\newtheorem{lemma}[theorem]{Lemma}
\newtheorem{notation}[theorem]{Notation}

\newtheorem{proposition}[theorem]{Proposition}
\newtheorem{remark}[theorem]{Remark}

\numberwithin{equation}{section}

\begin{document}

\title{Topologies on unparameterised path space}

\author{Thomas Cass$^{\ast 1}$ and William F. Turner$^{\dagger 1}$}
\thanks{$^{\ast}$The work of Thomas Cass was supported by the EPSRC Programme Grant``DataSig" EP/S026347/1}
\thanks{$^{\dagger}$ The work of William F. Turner was supported by the EPSRC Centre for Doctoral Training in Mathematics of Random Systems: Analysis, Modelling and Simulation (EP/S023925/1)}
\address{$^1$Department of Mathematics, Imperial College London}
\date{\today}
\begin{abstract}
The signature of a path, introduced by K.T. Chen \cite{Chen1} in $1954$, has been extensively studied in recent years. The fundamental $2010$ paper \cite{HL} of Hambly and Lyons showed that the signature is an injective function on the space of continuous, finite-variation paths up to a general notion of reparameterisation called tree-like equivalence. More recently, the approximation theory of the signature has been widely used in the literature in applications. The archetypal instance of these results, see e.g. \cite{LLN}, guarantees uniform approximation on compact sets of a continuous function by a linear functional of the signature. 

In this paper we study in detail, and for the first time, the properties of three natural candidate topologies on the set of unparameterised paths, i.e. the tree-like equivalence classes. These are obtained by privileging different properties of the signature and are: (1) the product topology, obtained by equipping the range of the signature with the (subspace topology of the) product topology in the tensor algebra and then requiring $S$ to be an embedding, (2) the quotient topology derived from the 1-variation topology on the underlying path space, and (3) the metric topology associated to $d\left( \left[ \gamma\right] ,\left[ \sigma\right] \right) := \left\vert \left\vert \gamma^{\ast}-\sigma^{\ast}\right\vert
\right\vert_{1}$ using the (constant-speed) tree-reduced representatives $\gamma^{\ast}$ and $\sigma^{\ast}$ of the respective equivalence classes. We prove that the respective collections of open sets are ordered by strict inclusion, (1) being the weakest and (3) the strongest. Our other conclusions can be summarised as follows. All three topological spaces are separable and Hausdorff, (1) being both metrisable and $\sigma$-compact, but not a Baire space and hence being neither Polish nor locally compact. The completion of (1), in any metric inducing the product topology, is the subspace $G^{*}$ of group-like elements. The quotient topology (2) is \emph{not} metrisable and the metric $d$ is \emph{not} complete. We also discuss some open problems related to these spaces.

We consider finally the implications of the selection of the topology for uniform approximation results involving the signature. A stereotypical model for a continuous function on (unparameterised) path space is the solution of a controlled differential equation. We thus prove, for a broad class of these equations, well-definedness and measurability of the (fixed-time) solution map with respect to the Borel sigma-algebra of each topology. Under stronger regularity assumptions, we further show continuity of this same map on explicit compact subsets of the product topology (1). We relate these results to the expected signature model of \cite{LLN}.
\end{abstract}
\maketitle

\section{Introduction}
A continuous bounded variation path $\gamma:[0,1] \mapsto V$ taking values in a finite-dimensional vector space $V$ has associated to it a sequence of tensors in the product space $\prod_{i=0}^{\infty} V^{\otimes i}$ obtained by taking iterated integrals of successively higher degree. The resulting sequence, its signature $S(\gamma)$, was first studied by K.T. Chen \cite{Chen2, Chen1} and has,  in more recent work \cite{HL}, been shown to characterise $\gamma$ up to an equivalence relation on the space of continuous bounded variation paths called tree-like equivalence.  The signature map $S$ accordingly becomes well-defined and injective on the $\sim_{\tau}$-equivalence classes: the space of so-called unparameterised paths $\mathcal{C}_1$. 

The significance of this body of theory has been given additional impetus by recent applications, see e.g. \cite{LLN}, \cite{Handwriting} in which the invariance of the signature under $\sim_{\tau}$ provides a form of dimensional reduction. These applications benefit from wider properties of the signature, particularly the fact that the collection of monomials on the range $\mathcal{S}$ of the signature map coincides with the restriction to $\mathcal{S}$ of linear functionals on the tensor algebra, see \cite{LLN}. This allows an approximation theory using the signature to be developed, the most basic form of which provides that any continuous function $f$ on a (locally) compact subset $K$ of $\mathcal{C}_1$ can be uniformly approximated by (the restriction of) a linear functional on the tensor algebra. The choice of topology on the space of unparameterised paths, necessary to a complete understanding of this method, is often left unspecified in applications or is otherwise chosen to suit the application at hand. An instance of this is the paper \cite{CO}, where the authors prove universality and characteristicness for kernels derived from the signature, and in doing so they impose the quotient topology derived from the variation topology on the underlying space of parameterised paths.

The purpose of this paper is to broaden this discussion by evaluating and comparing the properties of three topologies which are arrived at by leveraging different properties of $\mathcal{C}_{1}$. The qualities we emphasise are:
\begin{enumerate}
        
    \item The injectivity of the signature map from $\mathcal{C}_{1}$ onto a subset of the product space $\prod_{i=0}^{\infty} V^{\otimes i}$. This allows one to equip the range of the signature map with a (subspace) topology, and hence to transfer this topology onto $\mathcal{C}_{1}$ by requiring the signature map to be an embedding. We refer to this as the product topology.  
    \item The fact that tree-like equivalence on the space of continuous bounded variation paths is an equivalence relation means that $\mathcal{C}_{1}$ can be endowed with the quotient topology derived from the 1-variation topology as in the reference \cite{CO} referred to above.
    \item The existence, for each equivalence class $[\gamma]$ in $\mathcal{C}_{1}$, of a so-called tree-reduced representative \cite{HL}, which is characterised by having minimal length. By considering the constant speed parameterised version of this representative $\gamma^{\ast}$, we can define the metric topology on $\mathcal{C}_{1}$ induced by
\begin{equation*}
d\left( \left[ \gamma\right] ,\left[ \sigma\right] \right) := \left\vert \left\vert \gamma^{\ast}-\sigma^{\ast}\right\vert
\right\vert_{1}. 
\end{equation*}
\end{enumerate}
The product and quotient topologies can be seen as opposite extremes for how one can topologise $\mathcal{C}_1$. The former utilises only the topology of the range of the signature map, while the latter relies only on the topology of the domain. The metric topology lies somewhere in the middle, defined through both the topology on the domain and properties of the signature. Variants of (1) may also be of interest, e.g. where a topology on $\mathcal{C}_{1}$ is induced from a subspace topology on the tensor algebra, but the principle remains the same \cite{C}. We focus on the product topology as it is the weakest one in which all projections of the signature map are continuous.
Given the central role played by uniform approximation to a now-growing body of applications, it seems important to establish the basic topological features of these approaches. The main conclusions of Section \ref{sec: top} are captured in the following list:
\begin{itemize}
    \item All of the three topologies listed above are separable and Hausdorff.
    \item The collections of open sets are strictly ordered by inclusion, the product topology being the weakest and the metric topology being the strongest.
    \item The product topology is metrisable and $\sigma-$compact, but not a Baire space, and so neither Polish nor locally compact.
    \item A completion of the product topology is given by the subspace $G^{*} \subset \prod_{i=0}^{\infty} V^{\otimes i}$ of group-like elements.
    \item The quotient topology is not metrisable.
    \item The metric $d$ is not complete.
\end{itemize}
The condition of metrisability is stated as an assumption in a number of recent works \cite{BLO, persistence, LG} and our results would therefore seem to preclude the use of quotient topology in these cases. A careful reading of the reference \cite{Giles} underpinning these results, however, shows the property of complete regularity and not metrisability is the key assumption. A relevant resulting question, which we have not been able to resolve, is whether the quotient topology is completely regular.  We will discuss this and related points in more detail throughout the text.

The uniform approximation theory referred to above invites the study of two points. The first is the availability of compact subsets of $\mathcal{C}_{1}$ --  which should, ideally, be explicitly describable -- and the second is to understand the classes of continuous functions on these subsets. With respect to both points, the underlying topology has an unignorable bearing. We investigate this in the final section of the paper proving that:
\begin{itemize}
    \item The subset $B(r)$ of unparameterised paths with tree reduced length bounded by $r< \infty$ is a compact subset of $\mathcal{C}_{1}$ in the product topology.  
    \item The fixed-time solution of a differential equation $dy_{t}=f(y_{t})d\gamma_{t}$, under suitable regularity and growth conditions on $f$, induces a well-defined function on $\mathcal{C}_{1}$. This function is continuous on $B(r)$.
\end{itemize}
The relevance of these results is accounted for by the uses of the signature; physical limits on the ability to record and store data impose a bound on the (tree-reduced) lengths of the paths that can practically be considered. From this point of view, restricting attention to functions on $B(r)$ is a natural step. The second item ensures that a rich class of input-output response pairs are admissible as the true underlying causal relationship, e.g. in regression analysis.  

From the perspective of doing probability on $\mathcal{C}_{1}$, it can be desirable to work with measures defined on the Borel $\sigma$-algebra of a Polish topology; see e.g. the monograph \cite{G} for a detailed overview. The results above exclude having this structure for both the product and the quotient topologies. Nevertheless, the $\sigma$-compactness of the product topology still offers a way to obtain some of the benefits of Polishness, notably versions of Ulam's and Lusin's Theorems still hold.  In Section \ref{sec: diffeq} we illustrate how this approach can be used to validate and extend the framework of the expected signature model proposed in \cite{LLN}. If a Polish space is genuinely needed, then an alternative is to consider completions. In the case of the product topology, this leads to the subspace of group-like elements $G^{*}$.

\subsection{More general unparameterised paths}
We comment here briefly on the notion of unparameterised paths that we adopt. In this article, as is customary in the literature, see for example \cite{CF, CO}, we work with paths over a fixed pre-defined interval (which we take to be $[0,1]$ for convenience). A plausible alternative might be to adapt our definitions to take account for paths defined on possibly different compact subintervals of $\mathbb{R}$. While the results for the product topology would remain unchanged, this would result in a different topology for the domain of the signature map and the conclusions for the quotient and metric topologies would need further definitions to be interpreted. Technical obstacles for instance prevent the easy construction of a $1-$variation type distance between a path defined on an interval $[a,b]$ and another defined on $[c,d]$\footnote{One approach is to consider reparameterisations (non-decreasing surjections) of the interval $[a,b]$ onto  $[c,d]$. However, the inverse of a reparameterisation is not guaranteed to be a reparameterisation. Additionally, whilst every path is a continuous reparameterisation of itself run at constant speed, the reverse does not necessarily hold. Combined, these problems pose difficulties in defining a distance that is symmetric and satisfies the triangle inequality}. 
\section{Signatures and unparameterised paths}
We work with spaces of paths defined on the closed interval $[0,1]$ taking values in a finite-dimensional vector space $V$. We will use $|\cdot|$ to denote a fixed but arbitrary norm which we assume to be derived from an inner-product $\left\langle\cdot,\cdot \right\rangle$ on $V$. The properties of interest to us are invariant under translation by a constant vector; as such, we will assume throughout that all paths start at the origin $0 \in V$. The notion of finite $p$-variation is well known.
\begin{definition}
Let $1\leq p < \infty$. We denote by $C_{p}$ the space of paths $\gamma:[0,1]\mapsto V$ such that $\gamma_{0}=0$ and which have finite $p$-variation in the sense that
\begin{equation}
\left\vert \left\vert \gamma\right\vert \right\vert _{p}=\left( \sup_{D=\{t_{i}\}} \sum_{i} \left| \gamma_{t_{i}}-\gamma_{t_{i-1}} \right|^{p} \right)^{1/p} < \infty.
\end{equation}
\end{definition}
The signature of a path will be a central object in the later discussion.

\begin{definition}
Let $1\leq p < 2$ and assume that $\gamma$ belongs to $C_{p}$. The signature of $\gamma$ is defined to be the element of the product space
\begin{equation*}
S(\gamma) = (1,S_{1}(\gamma),S_{2}(\gamma),\dots) \in \prod_{i=0}^{\infty}V^{\otimes i},
\end{equation*}
where for every $n=0,1,2,\dots$ the expression $S_{n}(\gamma)$ denotes the $n$-fold iterated Young integral 
\begin{equation*}
S_{n}(\gamma)=\int_{0<t_{1}<t_{2}<\dots<t_{n}<1}d\gamma_{t_{1}} d\gamma_{t_{2}} \dots d\gamma_{t_{n}} \in V^{\otimes n}.  
\end{equation*}
\end{definition}
\begin{notation}
For $1\leq p <2$ we let $\mathcal{S}_{p}$ denote the subset $S(C_{p}) \subset \prod_{i=0}^{\infty}V^{\otimes i}$, the image of the signature map. When $p=1$ we write $\mathcal{S}_{1}=\mathcal{S}$.
\end{notation}
We let $\pi_{n}:\prod_{i=0}^{\infty}V^{\otimes i} \mapsto \prod_{i=0}^{n}V^{\otimes i}$ denote the canonical projection and we call $S^{(n)}(\gamma):=\pi_{n}S(\gamma)$ the $n$-step truncated signature of $\gamma$. The map $S^{(n)}:C_p\to\prod_{i=0}^{n}V^{\otimes i}$ is continuous, see \cite[Corollary 2.11]{SaintFlour}. The set $\prod_{i=0}^{\infty}V^{\otimes i}$ becomes an algebra, the tensor algebra, when equipped with an appropriate collection of vector space operations and an associative product (the tensor product, which we shall write implicitly). The definition is standard and we do not repeat it; the reader is referred to \cite{reut}. As is usual, we use $T((V))$ when we wish to emphasise the role of the algebra structure in addition to the underlying set $\prod_{i=0}^{\infty}V^{\otimes i}$. The identity element of this algebra is
\[
\mathbf{1}:=(1,0,0,\dots)\in \prod_{i=0}^{\infty}V^{\otimes i}.
\]
\begin{definition} 
The set of group-like elements $G^{*}$ is the subset of $\prod_{i=0}^{\infty}V^{\otimes i}$ defined by 
\[
G^{*}=\{\mathbf{x} \in T((V)): \text{ $\forall n$ } \pi_{n}\mathbf{x} \in S^{(n)}(C_{1})  \}.  
\]
\end{definition}
\begin{remark} In other words, $\mathbf{x}$ is in $G^{*}$ if and only if every projection $\pi_{n}\mathbf{x}$ can be realised as the $n$-step truncated signature of a continuous bounded variation path.
There are alternative ways to characterise $G^{*}$. For example, if $\exp:T((V))\rightarrow T((V))$ is the exponential map with respect to the tensor product on $T((V))$, then
$G^{*}=\exp(\mathfrak{g})$, where $\mathfrak{g}$ is the Lie algebra generated by $V$. See Theorem 3.2 of \cite{reut} for further equivalent algebraic characterisations.
\end{remark}
\begin{lemma} \label{polish}
The product space $\prod_{i=0}^{\infty}V^{\otimes i}$ endowed with the product topology is a Polish space. The set $G^{*}$ is closed in $\prod_{i=0}^{\infty}V^{\otimes i}$ with respect to this topology and $\mathcal{S} \subseteq G^{*} = \bar{\mathcal{S}}$.
\end{lemma}
\begin{proof}
A countable product of Polish spaces is a Polish space from general theory. To see that $G^{*}$ is closed we show that the limit $\mathbf{x}$ of any convergent sequence $(\mathbf{x}_{m})_{m=1}^{\infty}$ in $G^{*}$ is again in $G^{*}$. Let $\mathbf{x}_{m}=(x_{m,0},x_{m,1},\dots,x_{m,j},\dots)$, the convergence $ \lim_{m\rightarrow \infty} \mathbf{x}_{m}= \mathbf{x}$ in the product topology holds if and only if $x_{m,j}\rightarrow x_{j}$ in $V^{\otimes j}$ as $m \rightarrow \infty$ for every $j=0,1,2,\dots$ For any $n$ its follows that $\pi_{n}\mathbf{x}_{m} \rightarrow \pi_{n}\mathbf{x}$ as $m \rightarrow \infty$, and since it is known that $G_{n}=\pi_{n}(G^{*})$ is closed in the truncated tensor algebra, see for example \cite[Theorem 7.30]{FV} or \cite[Proposition 2.25]{SaintFlour}, it follows that $\pi_{n}\mathbf{x}$ is also in $G_{n}$. By definition this implies $\mathbf{x}$ is then in $G^*$. To see the final statement, it is clear by definition that $\mathcal{S} \subseteq G^{*}$. Then if $\mathbf{x}$ is in $G^{*}$, for every $n$ there exists by definition $\gamma_{n}$ such that $S^{(n)}(\gamma_{n})=\pi_{n} \mathbf{x}$. The sequence determined by $\mathbf{x}_{n}=S(\gamma_{n})$ in $\mathcal{S}$ then converges to $\mathbf{x}$ in $\prod_{i=0}^{\infty}V^{\otimes i}$. It follows that $G^{*}=\bar{\mathcal{S}}$.
\end{proof}
\begin{remark}
It is well known that the inclusion $\mathcal{S}\subset G^*$ is in fact strict, and the results of Section \ref{sec: top} provide a new (non-constructive) proof.
\end{remark}
The tensor algebra plays a special role in the study of the signature because of the way it interacts with a natural binary operation on path space. Recall that concatenation $*:C_{p}\times C_{p} \to C_{p}$ is the binary operation defined by 
\[
(\gamma*\sigma)_{t} =
\begin{cases}
\gamma_{2t}, & \text{for } 0\leq t\leq \frac{1}{2}\\
\sigma_{2t-1}+\gamma_{1}, & \text{for } \frac{1}{2}\leq t\leq 1.
\end{cases}
\]
Another important operation is the unary operation $\overleftarrow \ : C_{p} \to C_{p}$ which reverses the order of the path: 
\[
\overleftarrow{\gamma}_{t}:=\gamma_{1-t}-\gamma_{1} 
\]
The following relations hold \cite{Chen1}:
\begin{equation} \label{sig ops}
S(\gamma*\sigma)=S(\gamma)S(\sigma), \text{ and } 
\mathbf{1}=S(\overleftarrow{\gamma})S(\gamma)=S(\gamma)S(\overleftarrow{\gamma}),
\end{equation}
where multiplication is in $T((V))$.

Distinct paths can have equal signatures. For example, for any $\gamma \neq o$ we have that $\gamma*\overleftarrow{\gamma} \neq o$, while from (\ref{sig ops}) 
\[
S(o)=\mathbf{1}=S(\gamma*\overleftarrow{\gamma}).
\]
More generally, suppose $\tau:[0,1]\mapsto [0,1]$ is any parameterisation of $[0,1]$, i.e. a continuous, non-decreasing surjection of the interval $[0,1]$ onto itself. Then $S(\gamma\circ \tau)=S(\gamma)$, namely the signature is invariant under reparameterisation. For $1\leq p< 2$ we can define an equivalence relation on $C_{p}$ by identifying paths with the same signature. One of the main results of \cite{HL} is to show that this coincides with the notion of tree-like equivalence and hence to give a complete description of the equivalence classes. We recall that a path $\gamma$ is tree-like if there exists a real tree $\mathcal{T}$ such that $\gamma$ admits a factorisation $\gamma=\phi \circ \rho$ through a pair of continuous maps $\rho:[0,1]\mapsto \mathcal{T}$ and $\phi:\mathcal{T} \mapsto V$, where $\rho(0)=\rho(1)$. We then have the following definition. 
\begin{definition}[Tree-like equivalence]
Let $1\leq p <2$ and assume that $\gamma$ and $\sigma$ are in $C_{p}$. If $\gamma*\overleftarrow{\sigma}$ is tree-like then we call $\gamma$ and $\sigma$ tree-like equivalent and write $\gamma \sim_{\tau} \sigma$.  
\end{definition}
In \cite{HL} for the case when $p=1$, and later in \cite{BGLY} for the case $p>1$, which even covers weakly geometric rough paths, the following collection of results are proved.
\begin{theorem}[\cite{HL}, \cite{BGLY}]
\label{HL}Let $1 \leq p <2 $, then 
\begin{enumerate}
\item The relation $\sim_{\tau}$ defines an equivalence relation on $C_{p}$ which we call tree-like equivalence.
\item For any $\gamma$ and $\sigma$ in $C_{p}$, it holds that $%
\gamma\sim_{\tau}$ $\sigma$ if and only if $S\left( \gamma\right) =S\left(
\sigma\right) .$
\item In the case $p=1$, each $\sim_{\tau}$ equivalence class contains an element of minimal length. This element is called the tree-reduced representative. It is unique up to its parameterisation.
\end{enumerate}
\end{theorem}
\begin{definition}
Let $\gamma$ in $C_1$, then we denote the $\sim_\tau$ equivalence class of $\gamma$ by $[\gamma]$. The quotient space $C_1/\sim_\tau=\left\{[\gamma]:\gamma\in C_1\right\}$ is denoted by $\mathcal{C}_1$.
\end{definition}
\subsection{Tree-reduced paths}
It will be useful for us to work with a concrete representative of an equivalence class. When $p=1$ the tree-reduced representative is a natural choice. It is only necessary to fix the parameterisation.
\begin{definition}
Let $\gamma$ in $C_{1}$. We will use $%
\gamma^{\ast}$ to denote the tree-reduced representative of $\left[ \gamma \right]$ parameterised at constant speed.
\end{definition}
\begin{remark}
If $\gamma\in C_1$ is parameterised at constant speed, then $\gamma$ is Lipshitz continuous with Lipshitz constant $\left\vert\left\vert\gamma\right\vert\right\vert_1$.
\end{remark}
We can manufacture simple examples of paths that are already tree-reduced by considering axis paths defined with respect to an orthonormal basis. Given $v$ in $V$ we let $\gamma_{v}$ denote the linear path in $C_{1}$ whose derivative $\gamma_v'\equiv v$ on $[0,1]$. 
\begin{example}\label{pwlinear}
For a collection of $v_{1},v_{2},\dots,v_{m}$ in $V$ the piecewise linear path $\gamma:[0,1] \mapsto V$ in $C_{1}$ defined by $\gamma=\gamma_{v_{1}}*\gamma_{v_{2}}*\dots*\gamma_{v_{m}}$
has length $L=|v_{1}|+|v_{2}|+\dots+|v_{m}|$. When parameterised at constant speed it is defined by
\begin{equation}\label{constant speed}
    \gamma_{t}'=L\hat{v}_{i}:=L\frac{v_{i}}{|v_{i}|} \text{ for $t$ in $[t_{i-1}, t_{i}]$}, 
\end{equation}
where $t_0=0$ and
\[
t_{i}=\frac{1}{L}\sum_{j=1}^{i} |v_{j}| \text{ for $i=1,..,m$}.
\]
\end{example}
From now on, unless stated otherwise, we assume that all piecewise linear paths as in Example \ref{pwlinear} are parameterised at constant speed. The following result is known more generally for $C^2$ paths \cite{HL} and irreducible piecewise linear paths \cite{LX}, but the arguments rely delicately on estimates of the hyperbolic development of the path. In the case of axis paths, it is possible to present a simplified and direct proof using elementary tools. The use of axis paths will be fundamental to our subsequent discussion, and so we include a proof for completeness.
\begin{lemma}
\label{tree-reduced}Let $\gamma:\left[ 0,1\right] \rightarrow V$ in $%
C_{1}$ be of the form $\gamma=\gamma_{v_{1}}*\gamma_{v_{2}}*\dots*\gamma_{v_{m}}$ as in Example \ref{pwlinear}. Assume further that for every consecutive pair of vectors $\left\langle v_{i},v_{i+1} \right\rangle=0$ for $i=1,\dots,m-1$. Then $\gamma$ is tree reduced.
\end{lemma}
\begin{proof}
We prove the result indirectly not through the definition of tree-like equivalence, but by using asymptotic properties of the signature. We recall that the projective tensor norm is the largest cross
norm on $V^{\otimes n}$. For a tensor $A$ in the algebraic $n$-fold tensor product space it is defined by
\[
\left\vert \left\vert A \right\vert \right\vert _{\vee,n}=\inf\left\{ \sum_{i=1}^{r} |u_{1}^{i}|\dots|u_{n}^{i}|: A=\sum_{i=1}^{r} u_{1}^{i}u_{2}^{i}\dots u_{n}^{i} \right\}. 
\]
It can be deduced that if $A=\sum_{i} u_{1}^{i}u_{2}^{i}\dots u_{2n}^{i}$ then 
\[
\left\vert \left\vert A \right\vert \right\vert _{\vee,2n}\geq \varphi(A):=\sum_{i}\left\langle u_{1}^{i},u_{2}^{i} \right\rangle \dots\left\langle u_{2n-1}^{i},u_{2n}^{i} \right\rangle. 
\]
Note here that $\varphi$ is a bounded linear map on $V^{\otimes n}$. Consider $\gamma$ with the constant-speed parameterisation described in Example \ref{pwlinear}. Then, using the previous inequality, we obtain the lower bound
\begin{align*}
(2n)!\left\vert \left\vert S_{2n}\left( \gamma\right) \right\vert \right\vert
_{\vee,2n} & \geq (2n)! \varphi\left(\int_{\Delta^{2n}\left[ 0,1\right] }\gamma_{t_1}^\prime\otimes\gamma_{t_2}^\prime\otimes\dots\otimes\gamma_{t_{2n}}^\prime dt_{1}dt_{2}\dots dt_{2n}\right)\\
& = (2n)! \int_{\Delta^{2n}\left[ 0,1\right] }\varphi\left(\gamma_{t_1}^\prime\otimes\gamma_{t_2}^\prime\otimes\dots\otimes\gamma_{t_{2n}}^\prime\right) dt_{1}dt_{2}\dots dt_{2n}\\
& \geq (2n)! \int_{\Delta^{2n}\left[ 0,1\right] }\prod
_{j=1}^{n}\left\langle
\gamma_{t_{2j-1}}^{\prime},\gamma_{t_{2j}}^{\prime}\right\rangle
dt_{1}dt_{2}\dots dt_{2n} \\ 
& = \mathbb{E}\left[ \prod_{j=1}^{n}\left\langle \gamma_{U_{\left( 2j-1\right)
}}^{\prime},\gamma_{U_{\left( 2j\right) }}^{\prime}\right\rangle \right]=:\mathbb{E}\left[X_{n}\right],
\end{align*}
where $0\leq U_{\left( 1\right) }\leq U_{\left( 2\right) }\leq\dots\leq
U_{\left( 2n\right) }\leq1$ are the order statistics of an i.i.d. sample of $%
2n$ uniform $\left[ 0,1\right] $ random variables defined on a probability space $(\Omega, \mathcal{F}, \mathbb{P})$. Let $N_{i}$, for $i=1,\dots,m$, denote the number of realisations in this sample which are contained in the interval $[t_{i-1},t_{i}]$. We define three events by
\[
A=\bigcap_{i=1}^{m} \left\{N_{i} \text{ is even} \right\},\  
B=\bigcup_{i=1}^{m} \left\{N_{i}=0 \right\},\ 
C=B^{c} \cap \bigcup_{i=1}^{m} \left\{N_{i} \text{ is odd} \right\}. 
\]
The event $B$ is required, since it is possible for $X_n$ to be non-zero even if one of the $N_i$ is odd: the order statistics may ``skip'' a segment. Their union is $\Omega$, and it can easily be checked that 
\[
X_{n} = \left\{\begin{array}{lr}
        L^{2n}, & \text{ on } A \\
        0, & \text{ on } C, \\
        \end{array} \right. 
  \]
and therefore that 
\begin{equation}\label{lb}
(2n)!\left\vert \left\vert S\left( \gamma\right) \right\vert \right\vert
_{\vee,2n} \geq L^{2n}\mathbb{P}(A)-L^{2n}\mathbb{P}(B).
\end{equation}
Letting $\{E_{i}:i=1,2,\dots,m\}$ be a collection of i.i.d Rademacher random variables and $r=\min_{i=1,\dots,m}|t_{i}-t_{i-1}|\in (0,1)$, simple arguments can be employed to show 
\[
\mathbb{P}(A)=\mathbb{E}\left[\left(\sum_{i=1}^{m} E_{i} \frac{|v_{i}|}{L} \right)^{2n} \right]\geq \frac{2}{2^{m}}, \text{ and } \\
\mathbb{P}(B) \leq m (1-r)^{2n}.
\]
Using these inequalities together with (\ref{lb}) we obtain easily that  
\[
\limsup_{n\rightarrow\infty} \left( n!\left\vert \left\vert
S\left( \gamma\right) \right\vert \right\vert_{\vee, n}\right)^{\frac{1}{n}}\geq L.
\]
For any path $\tilde{\gamma}$ in $C_{1}$ with length $\tilde{L}$ such that $S(\gamma)=S(\tilde{\gamma})$ we deduce 
\[
L \leq \limsup_{n\rightarrow\infty} \left( n!\left\vert \left\vert
S\left( \gamma\right) \right\vert \right\vert_{\vee, n}\right)^{\frac{1}{n}}\leq \tilde{L},
\]
where the second inequality follows from \cite[Proposition 2.2]{SaintFlour}, and hence $\gamma$ is tree-reduced.
\end{proof}
The following construction can be found in \cite{LX}. It illustrates how the class of paths considered in the previous lemma can be used to gain insight into properties of the signature. It will be an integral tool in the next section.
\begin{example} [\cite{LX}] \label{LyonsXu}
Let $W$ be a two-dimensional vector subspace of $V$ which is identified with $\mathbb{R}^{2}$ through an orthonormal basis $\{v_{i}:i=1,2\}$. We can define two sequences $\{\rho_{n}:n=1,2,\dots\}$ and $\{\sigma_{n}:n=1,2,\dots\}$ of so-called axis paths by $\rho_{1}:=\gamma_{v_{1}}*\gamma_{v_{2}}$,  $\sigma_{1}:=\gamma_{v_{2}}*\gamma_{v_{1}}$ and then for $n=2,3,\dots$ by 
\[
\rho_{n}=\rho_{n-1}*\sigma_{n-1} \text{ and }
\sigma_{n}=\sigma_{n-1}*\rho_{n-1}.
\]
If two consecutive line segments $v_i$ and $v_{i+1}$ are positively collinear, then by replacing $\gamma_{v_{i}}*\gamma_{v_{i+1}}$ with $\gamma_{v_i+v_{i+1}}$, we may write each path in the form required for Lemma \ref{tree-reduced}. By virtue of this Lemma, each of these paths is tree-reduced. For every $n$, the paths $\rho_{n}$ and $\sigma_{n}$ have length $2^{n}$ and it was further shown in \cite{LX} that the terms in their signatures coincide up to degree $n$, i.e.
\[
S_{k}(\rho_{n})=S_{k}(\sigma_{n}) \text{ for $k=1,2,\dots,n$ }.
\]
Moreover, it can be shown that $S_{n+1}(\rho_{n})$ and $S_{n+1}(\sigma_{n})$ are not equal for any $n$. Consequently, each tree-reduced path $\Gamma_k:= \sigma_k*\overleftarrow{\rho}_{k}$ has length $2^{k+1}$ and satisfies
\[
S_{m}\left( \Gamma_{k}\right)
=S_{m}\left(\sigma_k*\overleftarrow{\sigma}_{k}\right)=0 \text{ for  $m=1,2,\dots,k$, while } S\left( \Gamma_{k}\right) \neq\mathbf{1},
\]
since the projection at level $k$ of the algebraic inversion and multiplication depends only on the level $k$ projections of the inputs.
\end{example}

\section{Topologies on unparameterised path space}\label{sec: top}
Our aim now is to explore the basic properties of three topologies on the space $\mathcal{C}_{1}$ of unparameterised paths. This will allow us, in the next section, to put forward a framework within which one can understand uniform approximation results for continuous functions defined on compact subsets of $\mathcal{C}_{1}$.  As noted in \cite{HL}, there is no canonical topology, but there are at least three principled approaches to constructing one:
\begin{enumerate}
\item  To use the injectivity of the signature map $S:$ $\mathcal{C}_{1}\rightarrow\mathcal{S}{\subset}\prod_{i=0}^{\infty}V^{\otimes i}$. By requiring that $S$ is a topological embedding, any topology on $\mathcal{S}$ induces a unique topology on $\mathcal{C}_{1}$. We focus subsequently on the subspace topology on  $\mathcal{S}$ with respect to the product topology on $\prod_{i=0}^{\infty}V^{\otimes i}$. We denote this by $(\mathcal{C}_{1},\chi_{_{\text{pr}}})$ and refer to it as the product topology.
    \item To use the quotient topology on $\mathcal{C}_{1}$ inherited from the 1-variation norm topology on $C_{1}$. We denote this topological space by $(\mathcal{C}_{1},\chi_{_{\tau}})$.    
    \item To define a metric using the tree-reduced representatives parameterised at constant speed:%
\begin{equation}\label{metric}
d\left( \left[ \gamma\right] ,\left[ \sigma\right] \right) := \left\vert \left\vert \gamma^{\ast}-\sigma^{\ast}\right\vert
\right\vert_{1}. 
\end{equation}
It is easily seen that this defines a metric on $\mathcal{C}_{1}$. The associated metric topology will be denoted by $(\mathcal{C}_{1},\chi_{d})$.
\end{enumerate}
\begin{notation}
Let $\pi: C_{1} \to \mathcal{C}_{1}$ be the surjective map which takes $\gamma$ to its tree-like equivalence class $[\gamma]$.
\end{notation}
Let us now show some simple properties of these topological spaces.
\begin{lemma} \label{single}
Every singleton $\{[\gamma]\}$ is closed in $\chi_{_{\text{pr}}}$. 
\end{lemma}
\begin{proof}
Since $V$ is Hausdorff, then so is the product topology on $\prod_{i=0}^\infty V^{\otimes i}$. Hence singletons are closed in $\prod_{i=0}^\infty V^{\otimes i}$, and consequently singletons are closed in $\chi_{\text{pr}}$, since $S$ is a topological embedding.
\end{proof}
% $\left\{S(\gamma)\right\}$ is closed in the product topology and consequently closed in the subspace topology on $\mathcal{S}$. By definition, the signature map $S$ is continuous with respect to $\chi_{_{\text{pr}}}$, and so $S^{-1}\left(\left\{S(\gamma)\right\}\right)=\{[\gamma]\}$ is closed in $\chi_{_{\text{pr}}}$.

% By definition the preimage $\pi^{-1}\{[\gamma]\}$ is the set $[\gamma]$. The fact that  $[\gamma]$ is a closed in $C_{1}$ is immediate: all the terms of any sequence of paths in $[\gamma]$ will all have identical signatures, and for each $k=0,1,2,,$ the function $\gamma \mapsto S_{k}(\gamma)$ is continuous in the $1$-variation topology.  

The following proposition shows that the three topologies we have proposed are distinct. In fact, the collections of open sets are strictly ordered with $\chi_{_{\text{pr}%
}}$ defining the weakest and $\chi_{d}$ the strongest of the topologies.
\begin{proposition}
\label{topologies}We have the strict inclusions $\chi_{_{\text{pr}%
}}\subset \chi_{\tau}\subset\chi_{d}$.
\end{proposition}

\begin{proof}
Letting $\left\vert\left\vert\cdot\right\vert\right\vert_{V^{\otimes k}}$ be a family of norms on the tensor powers $V^{\otimes k}$, a subbase for $\chi_{\text{pr}}$ is given by the collection of sets $%
\mathcal{V}$ of the form%
\[
V_{\sigma,k,\epsilon}=\left\{ \left[ \gamma\right] :\left\vert \left\vert
S_{k}\left( \sigma\right) -S_{k}\left( \gamma\right) \right\vert \right\vert_{V^{\otimes k}}
<\epsilon\right\} \text{ for }k\in\mathbb{N\cup}\left\{ 0\right\} \text{ and 
}\epsilon>0. 
\]
For any such set, the preimage $\pi^{-1}\left( V_{\sigma,k,\epsilon}\right)
=U_{\sigma,k,\epsilon}:=\left\{ \gamma:\left\vert\left\vert S_{k}\left( \sigma\right)
-S_{k}\left( \gamma\right) \right\vert\right\vert_{V^{\otimes k}} <\epsilon\right\}$ is open in $C_{1}.$ It follows that $\mathcal{V\subset}\chi_{\tau}$ and hence that $%
\chi_{_{\text{pr}}}\subseteq\chi_{\tau}.$\ To see the strict inclusion we consider a two-dimensional subspace $W$ of $V$ spanned by two orthonormal vectors $\{v_{1},v_{2}\}$. In this basis we consider the sequence of axis paths $(\Gamma_{k})_{k=1}^{\infty}$ constructed in Example \ref{LyonsXu}.
% Let $\Gamma_{k}:=\sigma_{k}\ast\overleftarrow{\gamma}_{k}$
% for $k=1,2.,,$. The each $\Gamma_{k}$ is a tree-reduced path of length $2^{k+1}$ which satisfies
% \[
% S_{n}\left( \Gamma_{k}\right)
% =0 \text{ for  $n=1,2,\dots,k$, while } S_{k+1}\left( \Gamma_{k}\right) \neq0.
% \]
Define the set $A:=\left\{\left[ \Gamma_{k}\right]: k=1,2,\dots\right\} \subset%
\mathcal{C}_{1}$. Then $A$ is not closed in the product topology
since for every $m$ we have%
\[
S_{m}\left( \Gamma_{k}\right) \rightarrow0\text{ as }k\rightarrow\infty 
\]
while $A$ does not contain the equivalence class of tree-like paths $[o].$
On the other hand, $A$ is closed in the quotient topology because any
sequence $\left( \gamma_{n}\right) _{n=1}^{\infty}$ in $\pi^{-1}\left(
A\right) $ which converges to $\gamma$ in $C_{1}$ must be such
that $\sup_{n}\left\vert \left\vert \gamma_{n}\right\vert \right\vert
_{1}<\infty$. Using the fact that each $\Gamma_{k}$ is tree-reduced and $\left\vert
\left\vert \Gamma _{k}\right\vert \right\vert _{1}=2^{k+1}$ then shows that there exists 
$N$ such that $\left\{ \gamma_{n}:n\in%
%TCIMACRO{\U{2115} }%
%BeginExpansion
\mathbb{N}
%EndExpansion
\right\} $ is a subset of $\cup_{k=1}^{N}\left[ \Gamma _{k}%
\right] = \cup_{k=1}^{N} \pi^{-1}\left(\{[\Gamma_{k}]\}\right)$. This latter set is closed in $C_1$ by Lemma \ref{single} and the fact that $\chi_{_{\text{pr}%
}}\subseteq \chi_{\tau}$. The limit $\gamma$ is then also in this set which is a subset of $\pi^{-1}(A).$

To prove that $\chi_{\tau}\subseteq\chi_{d}$ we let $\left[ \gamma\right] $ be in $\mathcal{C}_{1}$ and prove that every open neighbourhood $N$ of $\left[ \gamma\right] $ in $\chi_{\tau}$ is also a neighbourhood of $\left[ \gamma\right]$ in $\chi_{d}$. By definition, the pre-image $\pi ^{-1}\left(
N\right) \supset\left[ \gamma\right] $ is an open subset of $C_{1}$. The tree-reduced
representative $\gamma^{\ast}$ belongs to $\left[ \gamma\right] $ and consequently there exists $\delta>0$ such that $B_{\left\vert
\left\vert \cdot\right\vert \right\vert _{1}}\left(
\gamma^{\ast},\delta\right) =\{\sigma\in C_{1}:\left\vert
\left\vert \gamma^{\ast}-\sigma\right\vert \right\vert _{1}<\delta\}\subset
\pi^{-1}\left( N\right) .$ The result then follows because 
\[
N\supset\left\{ \left[ \sigma\right] :\left\vert \left\vert \gamma^{\ast
}-\sigma\right\vert \right\vert _{1}<\delta\right\} \supseteq\left\{ \left[
\sigma\right] :\left\vert \left\vert \gamma^{\ast}-\sigma^{\ast}\right\vert
\right\vert _{1}<\delta\right\} =B_{d}\left( \left[ \gamma\right]
,\delta\right). 
\]
To prove the strict inclusion we find a neighbourhood of the constant path $[o]
$ of the form $B_{d}\left( \left[ o\right] ,\delta\right) $ such that every $%
\chi_{_{\tau}}$-neighbourhood of $\left[ o\right] $ has non-empty
intersection with the complement $B_{d}\left( \left[ o\right] ,\delta\right) ^{c}:=\mathcal{%
C}_{1}\setminus B_{d}\left( \left[ o\right] ,\delta\right) .$ To
do so, we consider as above a two-dimensional subspace of $V$ spanned by an orthonormal set $\{v_{1},v_{2}\}$. Using this basis we define a family of axis paths given by: 
\[
\gamma_{\epsilon} = \gamma_{\epsilon v_{2}}*\gamma_{v_{1}}*\gamma_{-\epsilon v_{2}}*\gamma_{-v_{1}} \text { for $\epsilon \geq 0$.} 
\]
For every $\epsilon>0$ the curve $\gamma_{\epsilon}$ is
tree-reduced by Lemma \ref{tree-reduced}. On the other hand, $\gamma_{0}$ is tree-like equivalent to the constant path $o.$ It
follows that 
\begin{equation} \label{d lb2}
d\left( \left[ \gamma_{0}\right] ,\left[ \gamma_{\epsilon}\right] \right)
=\left\vert \left\vert \gamma_{\epsilon}\right\vert \right\vert _{1}\geq2,  
\end{equation}
while an easy calculation using (\ref{constant speed}) shows that $ 
\left\vert \left\vert \gamma_{\epsilon}-\gamma_{0}\right\vert \right\vert
_{1}\leq 6\epsilon.$
Hence any $\chi_{_{\tau}}$-neighbourhood of $\left[ o\right] $
must contain $\left[ \gamma_{\epsilon}\right] $ for some $\epsilon>0$. In
view of (\ref{d lb2}), the set $B_{d}\left( \left[ o\right] ,\delta\right) $ can
never contain such a neighbourhood whenever $\delta<2.$
\end{proof}
The inclusions $\chi_{_{\text{pr}%
}}\subset \chi_{\tau}\subset\chi_{d}$ and the fact that $\chi_{_{\text{pr}%
}}$ is the subspace topology of a Hausdorff space immediately yields the following corollary.
\begin{corollary}
All three candidate topologies are Hausdorff.
\end{corollary}
\begin{proposition}\label{prop: group cts}
$(\mathcal{C}_{1},\chi_{_{\text{pr}}})$ with the operations $\left[ \gamma\right] \cdot\left[ \sigma\right] :=\left[\gamma\ast\sigma\right]$ and $\left[\gamma\right]^{-1}:=\left[\overleftarrow{\gamma}\right]$ forms a topological group. This is not the case for the metric topology.
\end{proposition}
We provide below a proof from first principles, however the continuity may also be seen from existing results in the literature. For example, by using that $\pi_n G^*$ is a closed Lie subgroup of the closed (with respect to the product topology) set $\pi_n\tilde{T}((V))$, where $\tilde{T}((V)):=\{\mathbf{a}\in T((V)):a_0=1\}$. The manifold topology of the Lie group is the same as the subspace topology of the product topology, for further details we refer the reader to \cite[Chapter 7]{FV}.
\begin{proof}
We first show continuity of the group operations on $(\mathcal{C}_1,\chi_{\text{pr}})$. For $\left[ \cdot\right] ^{-1}$, we define the following map on $\prod_{i=1}^\infty V^{\otimes i}$, which when restricted to $\mathcal{S}$ is the algebraic inversion:
\begin{align*}
\psi:\prod_{i=0}^\infty V^{\otimes i}&\longrightarrow \prod_{i=0}^\infty V^{\otimes i}\\
\mathbf{a} = (a_0,a_1,a_2,\dots)&\mapsto \mathbf{1} + \sum_{n=1}^\infty (0,-a_1,-a_2,\dots)^n.
\end{align*}
By definition of the subspace topology and the fact that $\mathcal{S}\subseteq\psi^{-1}\left(\mathcal{S}\right)$, it suffices to show continuity of $\psi$ on the tensor algebra. Moreover, by definition of the product topology it is enough to show continuity in each factor. Since the projection onto $V^{\otimes n}$ of $\psi(\mathbf{a})$ depends only on $\pi_n \mathbf{a}$, we may consider the factorisation
\[
\begin{tikzcd}
\prod_{i=0}^\infty V^{\otimes i} \arrow[d, "\pi_n"'] \arrow[rd, "\psi(\cdot)_n"] &               \\
\prod_{i=0}^n V^{\otimes i} \arrow[r, "\psi_n"']                                 & V^{\otimes n}
\end{tikzcd}
\]
Where $\psi_n(\mathbf{b}):=\psi\left(\pi_n^{-1}\mathbf{b}\right)_n$. By continuity of $\pi_n$ we only need to show continuity of $\psi_n$. This follows from the continuity of the tensor product, the canonical isomorphisms $V^{\otimes k}\otimes V^{\otimes l}\cong V^{\otimes k+l}$, and the continuity of addition. The continuity of group multiplication follows similarly.

We now prove the discontinuity of group multiplication with respect to the metric topology. Let $\{v_1,v_2\}$ be orthonormal vectors and define the two sequences of axis paths
\[
\rho_{n}=\gamma_{\frac{v_{2}}{n}}*\gamma_{v_{1}},\quad \sigma_n =\gamma_{-\frac{v_{2}}{n}}*\gamma_{-v_{1}} \text{ for $n \in \mathbb{N}$}.
\]

For each $n$, the path $\rho_n$ (\textit{resp.} $\sigma_n$) is tree reduced so that $\rho_n^* =\rho_n$ (\textit{resp.} $\sigma_n^*=\sigma_n$). Moreover the concatenation $\rho_n*\sigma_n$ is tree reduced and parameterised at constant speed. It follows that
\[
d([\rho_n]\cdot[\sigma_n]),[o])= \left\vert\left\vert \rho_n*\sigma_n-o\right\vert\right\vert_1\geq 2.
\]
On the other hand, for every $n$
\[
d([\rho_n],[\gamma_{v_1}])=d([\sigma_n], [\gamma_{-v_1}])\leq \frac{3}{n},
\]
and so $\rho_n$ (\textit{resp.} $\sigma_n$) converges to $[\gamma_{v_1}]$ (\textit{resp.} $[\gamma_{-v_1}]$) in $\chi_d$. But since $\gamma_{v_1}*\gamma_{-v_1}$ is tree-like, we have
\[
d([\gamma_{v_1}]\cdot[\gamma_{-v_1}]),[o])=d([o],[o])=0.
\]
Hence multiplication is not continuous with respect to $\chi_d$.
\end{proof}
\begin{remark}
The question of continuity of the group operations on $(\mathcal{C}_1, \chi_\tau)$ is of particular interest. Showing continuity would imply that the quotient topology is completely regular since every topological group is uniformisable (and every uniform space is completely regular).
\end{remark}

\subsection{Complete metrisability of candidate topologies?}
We start by answering this question in the negative for the product topology $(\mathcal{C}_{1},\chi_{\text{pr}})$.
\begin{proposition}\label{prop: Baire}
The product topology $(\mathcal{C}_{1},\chi_{\text{pr}})$ is not a Baire space.
\end{proposition}
\begin{proof}
The proof has two parts: we first show that every non-empty open set in $\chi_{\text{pr}}$ is unbounded with respect to the metric $d$, and then that $\mathcal{C}_{1}$ may be written as the countable union of closed sets that are bounded with respect to $d$.

Consider the two dimensional subspace of $V$ spanned by two orthonormal vectors $\{v_1,v_2\}$, and let $(\Gamma_k)_{k=1}^\infty$ be the sequence of axis paths defined in Example \ref{LyonsXu}. Each $\Gamma_k$ is tree reduced by Lemma \ref{tree-reduced}, thus 
\begin{equation*}
    d([o],[\Gamma_k])=2^{k+1}.
\end{equation*}
As shown in the proof of Proposition \ref{topologies}, $\left[\Gamma_k\right]$ converges to $[o]$ in $\chi_{\text{pr}}$, and so every open neighbourhood of $[o]$ contains infinitely many terms of the sequence $(\Gamma_k)_{k=1}^\infty$. Hence every open neighbourhood of $[o]$ is unbounded in $d$. Now let $[\gamma]\in\mathcal{C}_1$ and $U$ be any open neighbourhood of $\gamma$. Define the map
\begin{align*}
    \lambda: U&\longrightarrow \ \mathcal{C}_1\\
    [\sigma]&\longmapsto [\gamma]^{-1}\cdot[\sigma],
\end{align*}
which is continuous by Proposition \ref{prop: group cts} and bijective with its image $V:=\lambda(U)$. The inverse map
\begin{align*}
    \lambda^{-1}: V&\longrightarrow \ U\\
    [\sigma]&\longmapsto [\gamma]\cdot[\sigma],
\end{align*}
is also continuous by Proposition \ref{prop: group cts}. Therefore $\lambda(U)$ is an open neighbourhood of $[o]$. If $U$ were bounded in $d$ from $[o]$ by some constant $K>0$, then by construction of $\lambda$, the length of the tree-reduced representative of every $[\sigma]$ in $V$ is at most $2K$. This would imply that $V$ were bounded in $d$, a contradiction.

For any $r>0$, define the set 
\[
B\left( r\right) :=\left\{ \left[ \gamma\right] \in \mathcal{C}_{1} :\left\vert \left\vert \gamma^{\ast}\right\vert \right\vert _{1}\leq
r\right\}.  
\]
We will show later, in Proposition \ref{prop: B(r)}, that $B(r)$ is compact in $\chi_{\text{pr}}$, and so closed by the Hausdorff property. Since each $B(r)$ is bounded in $d$, it must have empty interior by the preceding. Writing
\begin{equation*}
 \mathcal{C}_1 = \bigcup_{r=1}^{\infty}B(r),
\end{equation*}
we see that $\mathcal{C}_1$ is the countable union of nowhere dense sets and so is not a Baire space.
\end{proof}
\begin{remark}
It can also be shown that every open set in $\chi_\tau$ is unbounded in $d$, though we do not include a proof here.
\end{remark}

As a consequence of the above, and the Baire Category Theorem, we obtain the following corollary.

\begin{corollary}
The product topology $(\mathcal{C}_{1},\chi_{\text{pr}})$ is not locally compact.
\end{corollary}
The following is one of the main results of this paper.

\begin{theorem}\label{thm: topologies}
The three topologies on the space $\mathcal{C}_{1}$ of unparameterised paths are separable and have the following properties.
\begin{enumerate}
    \item The product topology $(\mathcal{C}_{1},\chi_{\text{pr}})$ is metrisable, but not completely metrisable.
    \item The quotient topology $(\mathcal{C}_{1},\chi_{\tau})$ is not metrisable.
    \item The metric $d$ defined by $(\ref{metric})$ is not complete.
\end{enumerate}
\end{theorem}
\begin{proof}
We will show separability of $\chi_{d}$ and use the inclusions $\chi_{_{\text{pr}%
}}\subset \chi_{\tau}\subset\chi_{d}$ to conclude separability of $\chi_{_{\text{pr}}}$ and $\chi_\tau$. By \cite[Proposition 1.31, Corollary 1.35]{FV}, the space of absolutely continuous paths with respect to $\left\vert\left\vert\cdot\right\vert\right\vert_1$ is separable. Consider the subspace $A$ of $C_1$ of tree-reduced paths parameterised at constant speed; that is all representatives of equivalence classes seen by the metric $d$. Since every path parameterised at constant speed is Lipshitz, we may consider $A$ as a subspace of absolutely continuous paths. Since any subspace of a separable metric space is separable, $A$ is separable with respect to $\left\vert\left\vert\cdot\right\vert\right\vert_1$. This is equivalent to the separability of $\chi_{d}$.

Item 1 is a consequence of the definition of $\chi_{\text{pr}}$ and Proposition \ref{prop: Baire}. The space $\prod_{i=0}^{\infty}V^{\otimes i}$ with the product topology is a Polish space. It follows that $\chi_{\text{pr}}$ coincides with the metric topology of the restriction of any metric which  generates the product topology on $\prod_{i=0}^{\infty}V^{\otimes i}$. Since $\mathcal{C}_1$ with the product topology is not a Baire space, the Baire Category Theorem implies no metric generating $\chi_{\text{pr}}$ is complete.

 We will prove item 2 using the fact that any metrisable space must be both first countable and regular \cite{N}. By assuming first-countability we show regularity cannot hold.  Thus we assume that there exists $\left\{
U_{i}\right\}_{i=1}^{\infty} $ a countable neighbourhood basis of $[o]$ which, without loss of generality, is assumed to
consist of open sets. As before, we consider sequences of axis paths in a two-dimensional subspace of $V$ defined relative to two orthonormal vectors $\{v_{1},v_{2}\}$. The first sequence will consist of tree-like paths; for every $n \in \mathbb{N}$ we take \[\gamma_{n}:=\gamma_{nv_{1}}*\gamma_{-nv_{1}}= \gamma_{nv_{1}}*\overleftarrow{\gamma_{nv_{1}}}.
\]
Every $\gamma_{n}$ belongs to the equivalence class $[o]$ and hence
\[
r_{n}=\sup\left\{ r>0:B\left( \gamma_{n},r\right) \subset\pi^{-1}\left(
U_{n}\right) \right\}>0.
\]
We introduce a strictly decreasing sequence of positive real numbers $\{a_{n}\}_{n=1}^{\infty}$ by 
\[
a_{1}:=\frac{1}{2}r_{1},\text{ }a_{n+1}=\frac{1}{2}\min\left\{
r_{n+1},a_{n}\right\} \text{ for }n=2,3,\dots 
\]
so that $\pi\left( B\left( \gamma_{n},a_{n}\right)
\right) \subset U_{n}$ for every $n$. We let $\Gamma_{n}$ be the path
\[
\Gamma_{n}=\gamma_{\epsilon_{n}v_{2}}*\gamma_{nv_{1}}*\gamma_{-\epsilon_{n}v_{2}}*\gamma_{-nv_{1}}, \text{ where }
\\
\epsilon_{n}=\frac{a_{n}}{6}
\]
For every $n$, the path $\Gamma_{n}$ is tree-reduced by Lemma \ref{tree-reduced} and has length $\left\vert \left\vert \Gamma_{n}\right\vert \right\vert _{1}=2n+2\epsilon_{n}$. By the same argument in the proof of Proposition \ref{topologies} these properties yields that $F=\cup_{n=1}^{\infty}\left\{\left[ \Gamma_{n}\right]\right\} $ is closed in $(\mathcal{C}_{1},\chi_{_{\tau}})$. It is also clear that $\left[ o\right] \notin F$. On the other hand, we can obtain the estimate
\begin{equation}\label{bound1}
     \left\vert \left\vert \Gamma_{n}-\gamma_{n}\right\vert \right\vert
_{1} < 6\epsilon_{n}=a_{n}.
\end{equation}
The collection $\{U_{n}\}_{n=1}^{\infty}$ is a neighbourhood basis at $[o]$, and therefore any set $U$ in $%
\chi_{_{\tau}}$ which contains $[o]$ must contain $U_{k}$ as a subset for
some $k.$ Using the estimate (\ref{bound1}) above and the definition of the sequence $\{a_{n}\}_{n=1}^{\infty}$ we have that $\Gamma_{k}\in B\left(
\gamma_{k},a_{k}\right) $ and therefore $\left[ \Gamma_{k}\right] \in U_{k}.$
In other words, we have shown that any neighbourhood of $\left[ o\right] $ must have
non-empty intersection with $F$ and therefore $(\mathcal{C}_{1%
},\chi_{_{\tau}})$ cannot be regular.

To prove the final item we define another sequence of paths
\[
\rho_{n}=\gamma_{\frac{v_{2}}{n}}*\gamma_{v_{1}}*\gamma_{-\frac{v_{2}}{n}}*\gamma_{-v_{1}} \text{ for $n \in \mathbb{N}$}.
\]
Again $\rho_{n}$ is tree-reduced, i.e. $\rho^{*}_{n}=\rho_{n}$, and there exists a constant $c>0$ such that for any $1\leq n<m$ we have 
\[
d([\rho_{n}],[\rho_{m}])=\left\vert \left\vert \rho_{n}-\rho_{m}\right\vert \right\vert_1 \leq \frac{c}{n}
\]
so that $([\rho_{n}])_{n=1}^{\infty}$ is a $d$-Cauchy sequence. However, because $S_{k}(\rho_{n})\rightarrow 0$ as $n \rightarrow \infty$ for every $k\in \mathbb{N}$ and since $\chi_{\text{pr}} \subset \chi_{d}$, the only possible limit point of the sequence is $[o]$. This cannot happen owing to 
\[
d([o],[\rho_{n}]) \geq 2,
\]
which holds for all $n$.
\end{proof}

% \begin{problem}
% Is $(\mathcal{C}_{1},\chi_{_{\tau}})$ separable? Is $(\mathcal{C}%
% _{1},\chi_{d})$ separable?
% \end{problem}
As mentioned in the introduction, several recent references, see for example \cite{BLO, persistence, LG}, state metrisability as an assumption. This premise may be traced to \cite{CO}, which itself derives from an application of Theorems 2.4, 3.1, and 4.6 of Giles in \cite{Giles}. The results of Giles are refashioned as Theorem 2.6 in \cite{CO} to include metrisability as a condition. A careful examination of the underlying reference, shows that metrisability is not necessary and, indeed, that the first two assertions of Theorem 2.6 hold for any topological space. The third and final assertion of Theorem 2.6 relies on Theorem 4.6 of \cite{Giles}, which is proved under the hypothesis that the space is completely regular. It is thus of interest to know whether the quotient topology has this property (or, equivalently, if it is uniformisable). With respect to this question, our proof of non-metrisability of the quotient topology does not yield an answer, except that first countability would then imply a lack of (complete) regularity. Alternatively, since $\chi_\tau$ is Hausdorff, local compactness would imply complete regularity.

The non-metrisability of the quotient topology may present challenges when using it for advanced applications in probability and stochastic analysis. As pointed out in the monograph \cite{G}, many probability measures in practice are Borel measures on Polish spaces. This assumption greatly simplifies the analysis in many applications by circumventing the complexities of working with general topological spaces. The development of theory reflects this. Useful tools are available in this setting including the facts that finite Borel measures on Polish spaces are tight (Ulam's Theorem), any Borel measurable function on such a space is continuous on a large compact set (Lusin's Theorem), and a family of Borel measures is tight if and only if it is relatively compact in the topology of weak convergence (Prokhorov's Theorem). In the case of the product topology $\chi_{\text{pr}}$, Ulam's Theorem, Lusin's Theorem, and one direction of Prokhorov's Theorem still hold. In the next section we illustrate how these  methods can be used to support an analysis of the expected signature model of \cite{LLN}.

Nevertheless, it may still be desirable to identify a Polish topology on $\mathcal{C}_{1}$. The previous theorem excludes $ \chi_{\tau}$ as a possibility, while  $\chi_{\text{pr}}$ can be metrised by using any metric which induces the topology on the product space. Taking the completion of $\mathcal{C}_{1}$ with respect to any such metric is a direct way to generate a Polish space. Alternatively, the completion of $\mathcal{C}_1$ with respect to the metric $d$ will also generate a Polish space, should a stronger topology be desired. An interesting question we have not answered is whether $\chi_d$ is completely metrisable.
\begin{proposition} \label{polishsp}
Let $\rho$ be any metric on that induces the product topology on $\prod_{i=0}^{\infty}V^{\otimes i}$, and let $\rho\vert_{\mathcal{S}}$ denote its restriction to $\mathcal{S}$. Then the set of group-like elements $G^{*}$ equipped with the metric $\rho\vert_{G^*}$ is a valid $\rho\vert_{\mathcal{S}}$-completion of $(\mathcal{C}_{1},\chi_{\text{pr}})$.
\end{proposition}
\begin{proof}
This follows more or less immediately from Lemma \ref{polish}.
\end{proof}
\begin{definition} \label{comp}
We denote the topological space obtained in Proposition \ref{polishsp} by $\left(\bar{\mathcal{C}_{1}},\chi_{\text{pr}}\right)$, and the $d-$completion by $\left(\bar{\mathcal{C}_{1}},\chi_{d}\right)$.
\end{definition}
% \begin{notation}
% We will write $\left(\bar{\mathcal{C}_{1}},\chi_{\text{pr}}\right)$ and $G^*$ somewhat interchangeably.
% \end{notation}
In contrast to $(\bar{\mathcal{C}_{1}},\chi_{\text{pr}})$, it is not clear whether there is a natural space with which to identify the $d-$completion of $(\mathcal{C}_{1},\chi_{d})$.

\section{Uniform approximation, linear regression and the Expected Signature Model}\label{sec: diffeq}
The contemporary use of the signature in applications in regression analysis, and its wider use in statistical and machine learning applications, is in large part underpinned by the following fundamental result.
\begin{theorem}[The Fundamental Theorem of Uniform Approximation by the Signature, \cite{LLN}]\label{found}
Let $(\mathcal{C}_{1},\chi)$ be the space of unparameterised paths equipped with a topology $\chi$ for which $S:\mathcal{C}_1\to T((V))$ is continuous. Let $\Phi:\mathcal{C}_{1}\rightarrow \mathbb{R}$ be a continuous function on a compact subset $K$ of $\mathcal{C}_{1}$. Then for every $\epsilon>0$, there exists a linear functional $L$ on $T((V))$ such that 
\[
\sup_{[\gamma] \in \mathcal{C}_{1}}\left\vert \Phi([\gamma])-L(S([\gamma]))\right\vert < \epsilon.
\]
\end{theorem}
A version of this theorem has appeared in various guises in earlier work, the first we believe being \cite{LLN}. The choice of topology is however rarely addressed explicitly; an exception, as noted earlier, is the paper \cite{CO} where the quotient topology is selected. The matter of which space to work with in the context of Theorem \ref{found} seems to merit some consideration. Our aim is not to advocate for any particular choice, but to audit our selected choice of candidates. We will prioritise three questions:
\begin{enumerate}
    \item How readily can compact subspaces of the chosen topology be found? 
    \item Given a compact subspace, how easy is it to exhibit continuous functions on these subspaces? 
    \item To what extent do these pairs of sets (compact subspaces and their set of continuous functions) relate to the practical use of Theorem \ref{found}?
\end{enumerate}
We address these points below. Our focus will mainly be on the product topology where it is simplest to provide positive answers to these questions.
\subsection{Compactness}
The inclusion maps $\iota: (\mathcal{C}_{1}, \chi_{\tau})\rightarrow (\mathcal{C}_{1},\chi_{\text{pr}
})$  and $\iota: (\mathcal{C}_{1},\chi_{d}) \rightarrow (\mathcal{C}_{1}, \chi_{\tau})$ are continuous by Proposition \ref{topologies}; compact subsets of $(\mathcal{C}_{1},\chi_{\text{pr}
})$ will therefore be at least as plentiful as for the alternatives. The next result describes some explicit examples of such sets.   
\begin{proposition}\label{prop: B(r)}
For any $r>0$, the set 
\[
B\left( r\right) =\left\{ \left[ \gamma\right] \in \mathcal{C}_{1} :\left\vert \left\vert \gamma^{\ast}\right\vert \right\vert _{1}\leq
r\right\}  
\]
is a compact subspace of $(\mathcal{C}_{1}, \chi_{\text{pr}
})$.  In particular, $(\mathcal{C}_{1}, \chi_{\text{pr}
})$ is $\sigma$-compact.
\end{proposition}

\begin{proof}
The product topology is metrisable and hence it suffices to prove sequential
compactness. Let $\left( \left[ \gamma_{n}\right] \right) _{n=1}^{\infty}$
be a sequence in $B\left( r\right) $ then, by definition of the path $%
\gamma^{\ast}$ and the set $B\left( r\right) ,$ the sequence $\left(
\gamma_{n}^{\ast}\right) _{n=1}^{\infty}$ satisfies%
\begin{equation}
\left\vert \gamma_{n}^{\ast}\left( t\right) -\gamma_{n}^{\ast}\left(
s\right) \right\vert \leq r\left( t-s\right) \text{ for all }s\leq t\text{
in }\left[ 0,1\right] \text{ and }n\in%
%TCIMACRO{\U{2115} }%
%BeginExpansion
\mathbb{N}
%EndExpansion
;   \label{equi}
\end{equation}
that is, the paths are equicontinuous. The Arzela-Ascoli theorem gives a
uniformly convergent subsequence which we again call $\left(
\gamma_{n}^{\ast}\right) _{n=1}^{\infty}$ for convenience. We write $\gamma$
for the limit and note from (\ref{equi}) that $\left\vert \left\vert
\gamma\right\vert \right\vert _{1}\leq r$ and therefore $\left\vert
\left\vert \gamma^{\ast}\right\vert \right\vert _{1}\leq r$ so that $\left[
\gamma\right] $ is in $B\left( r\right) .$ Let $1<p<2,$ then a standard
inequality gives 
\[
\left\vert \left\vert \gamma_{n}^{\ast}-\gamma_{m}^{\ast}\right\vert
\right\vert _{p}\leq 2r ^{1/p}\left\vert \left\vert \gamma
_{n}^{\ast}-\gamma_{m}^{\ast}\right\vert \right\vert _{\infty}^{1-1/p} 
\]
allowing us to conclude that $\left( \gamma_{n}^{\ast}\right) _{n=1}^{\infty }
$ is a Cauchy sequence and thus convergent to $\gamma$ in $p $-variaton. For
each $m$ the map%
\[
C_{p}\ni\sigma\mapsto S_{m}\left( \sigma\right) \in V^{\otimes m} 
\]
is continuous and it therefore holds that $S_{m}\left( \gamma_{n}^{\ast
}\right) \rightarrow S_{m}\left( \gamma\right) $ as $%
n\rightarrow\infty$. We have shown that $\left[ \gamma_{n}\right] \rightarrow%
\left[ \gamma\right] $ in $B\left( r\right) $ as $n\rightarrow \infty$ in
the product topology.
\end{proof}
\begin{corollary}
The range of the signature $\mathcal{S}$ is measurable with respect to the Borel sigma-algebra on $G^*$.
\end{corollary}
\begin{proof}
The inclusion map $\iota:\mathcal{S}\to G^*$ is continuous, and so $B(r)$ is a compact subset of $G^*$ for every $r>0$. Since $G^*$ is metrisable, each $B(r)$ is closed and measurable. Thus $\mathcal{S}$ is a countable union of measurable sets, hence it too is measurable.
\end{proof}
\begin{remark}
It is not difficult to exhibit compact subsets $K\subset\mathcal{C}_{1}$ which are
not contained in any $B\left( r\right) ,$ for $r>0.$ For instance, by using
the example of Proposition \ref{topologies}, the set $%
K=\cup_{k=1}^{\infty}\left[ \Gamma_{k}\right] \cup\lbrack o]$ is compact
since any (non-trivial) sequence has a subsequence converging to $[o]$ in the product topology. On the other
hand, each of the tree-reduced paths $\Gamma_{k}$ has length $2^{k+1}.$
\end{remark}

An important example of a function on $\mathcal{C}_{1}$ is the
soluton of an ordinary differential equation. Suppose that $W$ is a
finite-dimensional vector space. Let $\mathcal{T}\left( W\right) $ denote
the space of smooth vector fields on $W$ and suppose that $f:V\rightarrow 
\mathcal{T}\left( W\right) $ is linear, then we can solve uniquely the
differential equation%
\begin{equation}
dy_{t}=f_{d\gamma_{t}}\left( y_{t}\right),\ \text{started at }y_{0}. 
\label{ode}
\end{equation}
To relate this to the signature we introduce the following notation.

\begin{notation}
Let $f:V\rightarrow\mathcal{T}\left( W\right) $ be linear and let $I:W\to W$ be the identity map. Let $D\left(
W\right) $ denote the space of smooth differential operators on $W$ and for $%
k\in%
%TCIMACRO{\U{2115} }%
%BeginExpansion
\mathbb{N}
%EndExpansion
$ let $f^{\left( k\right) }:V^{\otimes k}\rightarrow D\left( W\right) $
be the unique linear map which is determined by 
\[
v_{1}\dots v_{k}\mapsto f_{v_{1}\dots v_{k}}^{\left( k\right) }\varphi\left( y\right)
:=\left. f_{v_{k}}\left( \cdot\right) \circ\dots\circ f_{v_{1}}\left(
\cdot\right)\varphi(\cdot) \right\vert _{\cdot=y}, 
\]
for a suitable class of smooth test functions $\varphi$.
\end{notation}

\begin{remark}
The right-hand side is a $k^{\text{th}}$-order differential operator. If $%
\left( y_{1},\dots,y_{n}\right) $ denotes coordinates on $W$ and $f_{v}\left(
y\right) =f_{v}^{i}\left( y\right) \partial_{i}$ then the notation $%
f_{v_{1}\dots v_{k}}^{\left( k\right) }\varphi\left( y\right) $ condenses the
otherwise unwieldy expression 
\begin{equation}
f_{v_{1}\dots v_{k}}^{\left( k\right) }\varphi\left( y\right) =f_{v_{k}}^{i_{k}}\left(
y\right) \partial_{i_{k}}\left( f_{v_{k-1}}^{i_{k-1}}\left( y\right)
\partial_{i_{k-1}}\left( \dots\left( f_{v_{1}}^{i_{1}}
\partial_{i_{1}}\varphi(y)\right) \right) \right).   \label{condense}
\end{equation}
\end{remark}

We can prove the following.

\begin{proposition}\label{prop: ode}
Suppose that $W$ is a finite-dimensional vector space. Let $f:V\rightarrow 
\mathcal{T}\left( W\right) $ be linear and assume that there exists $C<\infty
$ such that for every $k\in%
%TCIMACRO{\U{2115} }%
%BeginExpansion
\mathbb{N}
%EndExpansion
$ 
\begin{equation}
\sup_{y\in W}\left\vert\left\vert f^{\left( k\right) }I\left( y\right) \right\vert\right\vert\leq C^{k}, 
\label{derivative bound}
\end{equation}
where $\left\vert \left\vert f^{\left( k\right) }I\left( y\right) \right\vert
\right\vert $ denotes the operator norm of the linear map $f^{\left(
k\right) }I\left( y\right) .$ Define a function (the It\^{o}-map) by 
\[
\Psi_{y_{0},f}:C_{1}\rightarrow W,\text{ }\gamma\mapsto y_{1} 
\]
where $y$ is the unique solution over $\left[ 0,1\right] $ of the
differential equation (\ref{ode}). Then

\begin{enumerate}
\item The function $\Psi_{y_{0},\alpha}$ is constant on every equivalence
class of $\sim_{\tau}$ and, for every $\gamma~$in $C_{1}$, $%
\Psi_{y_{0},\alpha}\left( \gamma\right) $ is given by the convergent series 
\begin{equation}
\Psi_{y_{0},\alpha}\left( \gamma\right)
=y_{0}+\sum_{k=1}^{\infty}f_{S_{k}\left( \gamma\right) }^{\left( k\right)
}I\left( y_{0}\right) .   \label{series}
\end{equation}
\item The It\^{o}-map is a well-defined function from $\mathcal{C}_{1}$ into $W.$ For every $r>0$, the restriction of this function to $%
B\left( r\right) $ is continuous with respect to the topology $\chi_{\text{pr}}$ on $\mathcal{C}_{1}.$
\end{enumerate}
\end{proposition}

\begin{remark}
From (\ref{condense}) we can see that condition (\ref{derivative bound})
will hold if the derivatives $D^{k}f$ of $f$ satisfy 
\[
\sup_{k\in%
%TCIMACRO{\U{2115} }%
%BeginExpansion
\mathbb{N}
%EndExpansion
}\sup_{y\in W}\left\vert\left\vert D^{k}f\left( y\right) \right\vert\right\vert\leq C. 
\]
\end{remark}

\begin{proof}
The signature is invariant on the tree-like
equivalence classes and hence so will be the function $\Psi_{y_{0},\alpha}$
once we prove (\ref{series}). To see this, we observe first by iterated use
of the change-of-variable formula that for any $N\geq1$ we have
\[
\Psi_{y_{0},\alpha}\left( \gamma\right) =\Psi_{y_{0},\alpha}^{N}\left(
\gamma\right) +\int_{0< t_{1}<\dots<
t_{N+1}<1}f_{d\gamma_{t_{1}}d\gamma_{t_{2}}\dots d\gamma_{t_{N}}d\gamma_{t_{N+1}}}^{\left(
N+1\right) }I\left( y_{t_{1}}\right) , 
\]
where 
\[
\Psi_{y_{0},\alpha}^{N}\left( \gamma\right)
:=y_{0}+\sum_{k=1}^{N}f_{S_{k}\left( \gamma\right) }^{\left( k\right)
}I\left( y_{0}\right). 
\]
Condition (\ref{derivative bound}) ensures 
\[
\left\vert \Psi_{y_{0},\alpha}\left( \gamma\right)
-\Psi_{y_{0},\alpha}^{N}\left( \gamma\right) \right\vert \leq\frac{%
C^{N+1}L_{\gamma}^{N+1}}{\left( N+1\right) !}, 
\]
so that indeed (\ref{series}) holds.

Abusing notation we still use $\Psi_{y_{0},\alpha}$ to denote the
well-defined function on $\mathcal{C}_{
1}$ defined by $%
\Psi_{y_{0},\alpha }:\left[ \gamma\right] \mapsto\Psi_{y_{0},\alpha}\left(
\gamma\right) .$ To prove continuity of $\Psi_{y_{0},\alpha}$ on $B\left(
r\right) $ we take a sequence $\left[ \gamma_{n}\right] \rightarrow\left[
\gamma\right] $ in $B\left( r\right) $ in the product topology, and then
note that for any $N\geq1$%
\[
\left\vert \Psi_{y_{0},\alpha}\left( \left[ \gamma\right] \right)
-\Psi_{y_{0},\alpha}\left( \left[ \gamma_{n}\right] \right) \right\vert
\leq\sum_{k=1}^{N}\left\vert f_{S_{k}\left( \gamma^{\ast}\right) }^{\left(
k\right) }I\left( y_{0}\right) -f_{S_{k}\left( \gamma_{n}^{\ast}\right)
}^{\left( k\right) }I\left( y_{0}\right) \right\vert +\frac{2C^{N+1}r^{N+1}}{%
\left( N+1\right) !}. 
\]
Letting \thinspace$n\rightarrow\infty$ and then $N\rightarrow\infty$ we
learn that $\Psi_{y_{0},\alpha}\left( \left[ \gamma_{n}\right] \right) $ $%
\rightarrow\Psi_{y_{0},\alpha}\left( \left[ \gamma\right] \right) $ as $%
n\rightarrow\infty.$
\end{proof}

\subsection{ Revisiting the expected signature model in regression
analysis}
The chief motivation of Theorem \ref{found} is to provide a theoretical justification for
the so-called expected signature model introduced by Levin, Lyons and Ni in \cite{LLN}. To
explore the relation between this model and the foundations developed above,
suppose that $\Gamma$ is a random variable taking values in the space
$\left(  \mathcal{S},\chi_{\text{pr}}\right)$. Let $Y$ be another random variable, defined on the same
probability space as $\Gamma$, which takes values in a finite dimensional
vector space $W$. Then the goal of regression analysis is to learn the
conditional expectation $\mathbb{E}\left[  \left.  Y\right\vert \Gamma\right]
$, where $Y$ is interpreted as the response of some system to the input
$\Gamma.$ Another way of saying this is that we want to approximate the
Borel-measurable function $f:\mathcal{S}\rightarrow W$ defined by
\begin{equation}
f\left(  g\right)  =\mathbb{E}\left[  \left.  Y\right\vert \Gamma=g\right]
.\label{cond exp}%
\end{equation}
Suppose the law of $\Gamma$, a Borel probability measure on $\mathcal{S},$ is denoted by $\mu.$ Then, as we will see in Corollary \ref{cor: lusin}, by a version of Lusin's Theorem \cite{bogachev}, the function $f$ in (\ref{cond exp}) is almost continuous in the sense
that for any $\delta>0$ there exists a compact set $K=K_{\delta}$ such that
$\mu\left(  \mathcal{S}\setminus K\right)  <\delta$ and such that
$f$ is continuous on $K.$ By Theorem \ref{found}, it
is then reasonable to adopt the model
\[
Y=L\left(  \Gamma\right)  +\epsilon,
\]
where $L$ is the restriction to $\mathcal{S}$ of a linear function from $T\left(
\left(  V\right)  \right)  $ to $W$ and $\epsilon$ is a $W$-valued random
variable satisfying $\mathbb{E}\left[  \left.  \epsilon\right\vert
\Gamma\right]  =0\,.$ This obviates the need to find an explicit compact set
and continuous function relating the independent and dependent variables. 

The typical case is where $\Gamma=S\left(\gamma\right)  $ for some
stochastic process $\gamma$ in $C_{1}$, so that the probability measure
$\mu=S_{\ast}\mu_{\gamma}$ is the push-forward of the law of $\gamma$ under
the signature map, and where $Y$ is the solution to an ordinary differential
equation driven by $\gamma.$ The following result describes conditions for the
well-definedness and measurability of the functions on $\mathcal{C}_{1}\rightarrow W$ which result from this construction.

\begin{proposition}\label{prop: meas}
Let $W$ be a finite-dimensional vector space. Suppose that $f:V\rightarrow
\mathcal{T}\left(  W\right)  $ is a linear function which satisfies the
following conditions:

\begin{enumerate}
\item For every $v$ in $V$, $f_{v}:W\rightarrow W$ is Lipschitz continuous.

\item For every $R>0$ there exists a finite positive $C=C\left(  R\right)  $
such that for every $k$ in $\mathbb{N}$ the following bound holds%
\[
\sup_{\left\vert y\right\vert \leq R}\left\vert \left\vert f^{\left(
k\right)  }I\left(  y\right) \right\vert \right\vert \leq C\left(  R\right)
^{k}.
\]

\end{enumerate}

Then the It\^{o} map  $\gamma\mapsto y_{1}$ derived from the ordinary
differential equation $dy=f_{d\gamma_{t}}\left(  y_{t}\right)  $ starting at
$y_{0}$ is invariant with respect to the tree-like equivalence relation on
$C_{1}$ and induces a well-defined Borel measurable function $\Psi_{y_{0},f}%
:\mathcal{C}_{1}\rightarrow W$, where $\mathcal{C}_1$ is equipped with any of the candidate topologies. Additionally, for $\chi_{\text{pr}}$ and $\chi_d$, there exists a Borel measurable function
$\bar{\Psi}_{y_{0},f}:\bar{\mathcal{C}_{1}}\rightarrow W$ which agrees with
$\Psi_{y_{0},f}$ on $\mathcal{C}_{1}$.
\end{proposition}

\begin{proof}
The ordinary differential equation has a unique solution for every $\gamma$ in
$C_{1}.$ A repetition of the same argument in the proof of Proposition \ref{prop: ode} leads
to the estimate
\[
y_{1}=y_{0}+\sum_{k=1}^{N}f_{S_{k}\left(  \gamma\right)  }^{\left(  k\right)
}I\left(  y_{0}\right)  +E_{N+1},~\text{where }\left\vert E_{N}\right\vert
\leq\frac{C\left(  \left\vert y_{0}\right\vert +L_\gamma
\right)  ^{N}L_\gamma^{N}}{N!},%
\]
which allows one to deduce invariance on the equivalence classes by taking the
limit $N\rightarrow\infty.$ To see that the resulting function $\Psi_{y_{0}%
,f}:\mathcal{C}_{1}\rightarrow W$  is  measurable with respect to
$\mathcal{B(C}_{1})$, the Borel sigma algebra of $\mathcal{C}_{1}  ,$ and $\mathcal{B}(W)$ we notice that it is the
pointwise limit of the functions $L_{N}\circ S$ where $L_{N}$ denotes the
(restriction of) the linear function%
\[
L_{N}:\left(  a_{0},a_{1},\dots\right)  \mapsto y_{0}+\sum_{k=1}^{N}f_{a_{k}%
}^{\left(  k\right)  }I\left(  y_{0}\right)  \in W
\]
to $\mathcal{S\,}.$ The existence of a measurable extension follows from the
fact that $(W,\mathcal{B}(W))$ is a Polish space; see \cite[Theorem 1]{shortt} noting that $\mathcal{B(C}_{1})=\left\{
A\cap\mathcal{C}_{1}:A\in\mathcal{B(\bar{C}}_{1})\right\}  .$
\end{proof}
A general version of Lusin's Theorem yields the following corollary for the product topology $\chi_{_{\text{pr}}}$.
\begin{corollary}\label{cor: lusin}
Consider the setup as in Proposition \ref{prop: meas}, and let $\mu$ be a Borel probability measure on $(\mathcal{C}_{1},\chi_{_{\text{pr}}})$. Then, for every $\delta>0$, there is a compact set $K\subset \mathcal{C}_1$ such that $\mu(\mathcal{C}_1\setminus K)<\delta$ and $\left.\Psi_{y_0,f}\right\vert_{K}$ is continuous.
\end{corollary}
\begin{proof}
Since $\mathcal{C}_1$ is $\sigma-$compact with respect to $\chi_{_{\text{pr}}}$, $\mu$ is tight. Separability of $(\mathcal{C}_{1},\chi_{_{\text{pr}}})$ and \cite[Proposition 7.2.2]{bogachev} then imply $\mu$ is a Radon measure. Measurability of $\Psi_{y_0,f}$ and Lusin's Theorem for finite Radon measures \cite[Theorem 7.1.13]{bogachev} shows the existence of compact sets with the desired property.
\end{proof}
\section*{Acknowledgements}
The authors would like to thank Terry Lyons and Christian Litterer for helpful discussions and suggestions on an earlier draft of this paper.
\printbibliography[heading=bibintoc]

\end{document}